\documentclass[reqno,11pt,oneside]{amsart}
\usepackage[utf8]{inputenc}
\usepackage{amsmath}
\usepackage{amsfonts}
\usepackage{amssymb}
\usepackage{graphicx}
\usepackage{verbatim}
\usepackage{mathrsfs}
\usepackage{upref,amsthm,amsxtra,exscale}
\usepackage{cite}
\usepackage{dsfont}
\usepackage[
  left=3.5cm,
  right=2.5cm,
  top=4cm,
  bottom=4cm
]{geometry}
\usepackage[colorlinks=true,urlcolor=blue,
citecolor=red,linkcolor=blue,linktocpage,pdfpagelabels,
bookmarksnumbered,bookmarksopen]{hyperref}
\usepackage{tikz}

\usepackage{fullpage}

\newtheorem{theorem}{Theorem}[section]
\newtheorem{corollary}[theorem]{Corollary}

\newtheorem{lemma}[theorem]{Lemma}
\newtheorem{proposition}[theorem]{Proposition}

\numberwithin{equation}{section}

\def\r{\mathbb{R}}
\def\rn{\mathbb{R}^n}
\def\z{\mathbb{Z}}
\def\z2{\mathbb{Z}_2}

\def\sn{\mathbb{S}^n}

\def\eps{\varepsilon}
\def\rh{\rightharpoonup}

\def\im{\int_{M}}
\def\idm{\int_{\partial M}}
\def\irn{\int_{\rn}}
\def\vp{\varphi}

\def\vr{\varrho}

\def\cA{\mathcal{A}}

\def\cC{\mathcal{C}}

\def\cE{\mathcal{E}}

\def\cN{\mathcal{N}}

\def\cP{\mathcal{P}}

\def\cZ{\mathcal{Z}}

\def\tilde{\widetilde}
\def\d{\,\mathrm{d}}
\def\dv{\,\mathrm{dv}}

\def\dist{\mathrm{dist}}

\begin{document}

\title[Sign-changing solutions to the Yamabe problem on manifolds with boundary]
{Sign-changing solutions to the Yamabe problem on manifolds with boundary}

\author{Mónica Clapp}
\address[Mónica Clapp]{Instituto de Matemáticas,
Universidad Nacional Autónoma de México,
Campus Juriquilla,
76230 Querétaro, Qro., Mexico}
\email{monica.clapp@im.unam.mx}

\author{Benedetta Pellacci}
\address[Benedetta Pellacci]{Dipartimento di Matematica e Fisica,
Universit\`a della Campania ``Luigi Vanvitelli'',
Viale Lincoln 5,
81100 Caserta, Italy}
\email{benedetta.pellacci@unicampania.it}

\author{Angela Pistoia}
\address[Angela Pistoia]{Dipartimento SBAI, Sapienza Universit\`a di Roma,
via Antonio Scarpa 16, 00161 Roma, Italy }
\email{angela.pistoia@uniroma1.it}

\thanks{A. Pistoia  is partially supported by  the MUR-PRIN-20227HX33Z
  ``Pattern formation in nonlinear phenomena'' and  partially by
  INDAM-GNAMPA project ``Problemi di doppia curvatura su variet\`a a
  bordo e legami con le EDP di tipo ellittico''. 
 B. Pellacci is partially supported by  the MUR-PRIN-20227HX33Z
  ``Pattern formation in nonlinear phenomena'' and  partially by
  INDAM-GNAMPA project ``Problemi di ottimizzazione in PDEs da modelli biologici''.}

 \begin{abstract}
	Let $(M, g)$ be a compact Riemannian manifold with boundary. The Yamabe problem concerning the existence of a metric conformally equivalent to $g$ having constant scalar curvature on $M$ and constant mean curvature on its boundary is equivalent, in analytic terms, to finding a  positive solution to a nonlinear boundary-value problem with critical growth. While the existence of positive solutions to this problem is by now well understood, the existence of sign-changing (nodal) solutions remains largely open.
 
In this work we establish the existence of least-energy sign-changing solutions when the manifold  is positive and the mean curvature of the boundary    is a non-negative constant. More precisely, we prove that if $n\ge7$ and $M$ has a nonumbilic boundary point, then the problem admits least-energy nodal solutions. 
Our approach is variational and relies on the analysis of suitable conformal invariants and sharp energy estimates.

\end{abstract}
\maketitle

\section{Introduction}

Let $(M,g)$ be a compact smooth Riemannian manifold with boundary of dimension $n\geq 3$.  We consider the problem
\begin{equation}\label{problem}
\begin{cases}
 -\Delta u + c_nRu= a_n|u|^\frac{4}{n-2}u &\text{on \ } M,\\
 \quad\frac{\partial u}{\partial\nu} + d_nhu= b|u|^\frac{2}{n-2}u &\text{on \ }\partial M,
\end{cases}
\end{equation}
where $\Delta=\Delta_g$ is the Laplacian, $R=R_g$ is the scalar curvature of $M$, $h=h_g$ is the mean curvature of its boundary and $\nu=\nu_g$ is the outward unit normal on $\partial M$ with respect to the metric $g$, $a_n:=n(n-2)$, $c_n:=\frac{n-2}{4(n-1)}$, $d_n:=\frac{n-2}{2}$ and $b\in\r$.

A positive solution $u$ to this problem gives rise to a conformal metric $\tilde{g}:=u^\frac{4}{n-2}g$ having constant scalar curvature $c_n^{-1}a_n$ on $M$ and constant mean curvature $d_n^{-1}b$ on its boundary.

When $\partial M=\emptyset$ the problem \eqref{problem} reduces to the well-known Yamabe problem, a higher dimensional analogue of the uniformization theorem in complex analysis. It was solved through the combined efforts of Yamabe \cite{yamabe}, Trudinger \cite{trudinger}, Aubin \cite{aubin} and Schoen \cite{schoen} who proved the existence of a positive solution $u$ to the equation
\begin{equation}\label{eq:yamabe}
-\Delta u + c_nRu=|u|^\frac{4}{n-2}u \qquad\text{on \ } M,
\end{equation}
for any closed  Riemannian manifold $M$ of dimension $n\geq 3$. A comprehensive analysis was later provided by Lee and Parker in \cite{lp}.

The natural extension \eqref{problem} of the Yamabe problem to manifolds with boundary was initially studied by Cherrier and Escobar. In \cite{che}, Cherrier gave a condition for the existence of a positive solution, while Escobar established its existence for a large class of manifolds, first in the case of minimal boundary $b=0$ in \cite{e1}, and then for small positive and negative values of $b$ in \cite{e4}. Subsequently, Han and Li showed in \cite{hl0,hl} that, in many cases, a positive solution exists for any $b\in\r$.

Regarding the existence of nodal solutions (with sign change), there are several results for the Yamabe equation \eqref{eq:yamabe}. When $M$ is the standard sphere $\mathbb{S}^n$, W.Y. Ding established the existence of infinitely many sign-changing solutions to \eqref{eq:yamabe} in \cite{ding}. Other types of nodal solutions on $\mathbb{S}^n$ were exhibited in \cite{dmpp1,dmpp2,c,cfs,mm}.  In \cite{fp} Fernández and Petean showed that there is a solution with precisely $\ell$ nodal domains for every $\ell\geq 2$. For closed manifolds with suitable symmetries existence of nodal solutions to \eqref{eq:yamabe} has been established in \cite{cf,cp}. Nodal solutions on products were obtained in \cite{h,p}.

On an arbitrary closed Riemannian manifold of dimension $n\geq 11$, which is not locally conformally flat, the existence of a nodal solution to the Yamabe problem \eqref{eq:yamabe} was shown by Ammann and Humbert in \cite{ah}. They introduced a conformal invariant, called  the second Yamabe invariant, defined in terms of the second eigenvalue of the conformal Laplacian, and showed that on every such manifold there exists a function $u$ that realizes this invariant. This function necessarily changes sign, so the expression $\tilde{g}:=|u|^{4/(n-2)}g$ is not a metric on $M$. In \cite{ah} it is called a \emph{generalized metric}. This result was extended in \cite[Theorem 1.2$(v)$]{cpt} to manifolds of dimension $10$ that satisfy an additional assumption. Every function that realizes the second Yamabe invariant has least energy among all sign-changing solutions. The blow-up behavior of sequences of least energy sign-changing solutions of to \eqref{eq:yamabe} was described by Premoselli and Vétois in \cite{pv}. They also showed in \cite{pv2} that in dimensions $3$ to $10$, the second Yamabe invariant is not attained for any metric on the sphere that is sufficiently close to the standard one. 

In the present work, we are concerned with the existence of nodal solutions to the problem \eqref{problem}.

We assume that $M$ is \emph{positive}, this means that the first eigenvalue satisfies
\begin{align} \label{eq:M positive}
\lambda_1(M)&:=\inf_{v\in H^1(M)\smallsetminus\{0\}}\frac{\mathscr{Q}v}{\im v^2\dv}>0,
\end{align}
where $\mathscr{Q}$ is the quadratic form on $H^1(M)$ given by
\begin{equation}\label{eq:quadratic form}
\mathscr{Q}u:=\im(|\nabla u|^2+c_nRu^2)\dv+\idm d_nhu^2\d \sigma,
\end{equation}
and $\dv$ and $\d \sigma$ are the Riemannian measures on $M$ and $\partial M$ induced by the metric $g$. 

Our main result is the following.

\begin{theorem} \label{thm:main}
If $b\geq 0$, $n\geq 7$ and $M$ has a nonumbilic point $\xi\in\partial M$, then the problem \eqref{problem} has a least energy sign-changing solution.
\end{theorem}

The strategy to prove Theorem \ref{thm:main} is as follows. If $u\in H^1(M)$ is a nontrivial solution to \eqref{problem}, then, multiplying that problem by $u$ and integrating by parts, shows that $u$ belongs to the set
$$\cN:=\Big\{u\in H^1(M): u\neq 0 \text{ and }\mathscr{Q} u=a_n\im|u|^\frac{2n}{n-2}\dv+b\idm |u|^\frac{2(n-1)}{n-2}\d \sigma\Big\}.$$
If $u$ changes sign, multiplying \eqref{problem} by the positive and negative parts $u^\pm$ of $u$ shows that $u$ belongs to
$$\cE:=\{u\in\cN:u^+\in\cN\text{ and }u^-\in\cN\}.$$ 
We define
$$\mu_1(M):=\inf_{u\in\cN}J(u)\qquad\text{and}\qquad\mu_2(M):=\inf_{u\in\cE}J(u),$$
where $J$ is the energy functional associated with problem \eqref{problem} defined in \eqref{eq:J} below. These are conformal invariants (see Proposition \ref{prop:mu is invariant}). The first coincides with the mountain-pass value of $J$, defined in \cite[p.812]{hl}. Han and Li showed in \cite{hl0,hl} that, for a large class of manifolds, this value is attained at a positive solution of \eqref{problem} for each $b\in\r$, thus extending earlier results by Escobar \cite{e1,e4}. When $b\geq 0$, our Theorem \ref{thm:mu_2}, stated below, gives a condition for $\mu_2(M)$ to be attained by a solution of \eqref{problem} that changes sign. Its Corollary \ref{cor:mu_2a} provides a criterion to verify this condition in terms of a linear combination of two test functions. We show that a positive least energy solution to \eqref{problem} and the test function introduced by Han and Li in \cite[(1.8)]{hl} fulfill this criterion. 

To prove Theorem \ref{thm:mu_2}, we first show that, if $b\geq 0$, a sufficiently small neighborhood of the cone of positive (and negative) functions in $H^1(M)$ is strictly positively invariant under the negative gradient flow of $J$. This allows us to generate a Palais-Smale sequence for $J$ at a positive distance from both cones. Then, in order to show that the sequence converges, we characterize the loss of compactness below a suitable energy level. The case $b<0$ cannot be treated using this approach because the invariance of the positive and negative cones is not guaranteed. The existence of a sign changing solution to \eqref{problem} for $b<0$ remains an interesting and challenging open problem.

For $b=0$ the existence of a nodal solution to problem \eqref{problem} was established by Ho and Pyo in \cite[Theorems 1.3,  1.5 and 1.6]{hp} whenever $n\geq 7$ and the Weyl tensor does not vanish identically on $\partial M$. A similar result for the scalar flat problem (obtained by replacing $a_n$ with $0$ and setting $b=1$ in equations \eqref{problem}) was proved by Ho, Lee and Shin in \cite[Theorems 3.1, 5.2 and 6.1]{hls}. However, we emphasize that the approach developed here differs from those adopted in the aforementioned works. Results on the existence of positive solutions in the scalar flat case can be found in \cite{e2,mn}.

The paper is organized as follows. In Section \ref{sec:setting} we describe the variational setting for problem \eqref{problem} and prove Theorem \ref{thm:mu_2}. In Section \ref{sec:test function} we compute the energy of the test function and prove Theorem \ref{thm:main}. The result that describes the loss of compactness is proved in the appendix.

\section{The variational setting}
\label{sec:setting}

We denote the standard norm of $u$ in $H^1(M)$ by
$$\|u\|:=\left(\im(|\nabla u|^2+u^2)\dv\right)^{1/2}$$
and set $p:=\frac{2n}{n-2}$ and $q:=\frac{2(n-1)}{n-2}$. These are the critical exponents for the Sobolev embedding $H^1(M)\hookrightarrow L^p(M)$ and the Sobolev trace embedding $H^1(M)\hookrightarrow L^q(\partial M)$, $u\mapsto u|_{\partial M}$. As shown by Escobar in \cite[Section 1]{e2}, if $M$ is positive, then
\begin{equation}\label{eq:escobar1}
\bar{\lambda}_1(M):=\inf_{\substack{u\in H^1(M) \\ u|_{\partial M}\neq 0}}\frac{\mathscr{Q}u}{\idm |u|^2\d \sigma}>0.
\end{equation}
Since $\lambda_1(M)>0$, $\bar{\lambda}_1(M)>0$ and the trace map $H^1(M)\to L^2(\partial M)$, $u\mapsto u|_{\partial M}$, is linear and continuous, there exist $C_1>0$ and $C_2>0$ such that
\begin{equation} \label{eq:Q_norm}
C_1\|u\|^2 \leq \mathscr{Q}u \leq C_2\|u\|^2\qquad\text{for all \ }u\in H^1(M).
\end{equation}

The solutions to the problem \eqref{problem} are the critical points of the $\cC^2$-functional $J:H^1(M)\to\r$ given by
\begin{equation}\label{eq:J}
J(u):=\frac{1}{2}\mathscr{Q}u - \frac{a_n}{p}\im |u|^p\dv - \frac{b}{q}\idm |u|^q\d \sigma,
\end{equation}
whose derivative at $u$ is
\begin{equation}\label{eq:derivative}
J'(u)v=\im(\nabla u\cdot\nabla v+c_nRuv)\dv+\idm d_nhuv\d \sigma - a_n\im |u|^{p-2}uv \dv - b\idm |u|^{q-2}uv\d\sigma,
\end{equation}
for every $v\in H^1(M)$. The nontrivial solutions to \eqref{problem} belong to the set
\begin{align*}
\cN&:=\{u\in H^1(M): u\neq 0 \text{ and }J'(u)u=0\} \\
&\,=\Big\{u\in H^1(M): u\neq 0 \text{ and }\mathscr{Q} u=F(u)\Big\},
\end{align*}
where 
\begin{equation}\label{eq:F}
F(u):=a_n\im |u|^p \dv + b\idm|u|^q\d\sigma.
\end{equation}
$\cN$ has the following properties.

\begin{lemma}\label{lem:nehari}
\begin{itemize}
\item[$(a)$]There exists $c_0>0$ such that $\|u\|\geq c_0$ for all $u\in\cN$.
\item[$(b)$]$\cN$ is a Hilbert submanifold of class $\cC^2$ of $H^1(M)$ and a natural constraint for $J$.
\item[$(c)$]If $u\in H^1(M)\smallsetminus\{0\}$, then there exists a unique $t_u\in(0,\infty)$ such that $t_uu\in\cN$. The function $t\mapsto J(tu)$ is strictly increasing in $[0,t_u]$ and strictly decreasing in $[t_u,\infty)$.
\end{itemize}
\end{lemma}

\begin{proof}
$(a):$ \ If $u\in\cN$, then, by \eqref{eq:Q_norm} and the Sobolev embeddings, there is $c_1>0$ such that
$$\mathscr{Q}u=a_n\im |u|^p \dv + b\idm|u|^q\d\sigma\leq c_1\Big((\mathscr{Q}u)^{p/2}+(\mathscr{Q}u)^{q/2}\Big).$$
Therefore, $\mathscr{Q}u\geq c_0>0$ for all $u\in\cN$, and the statement $(a)$ follows from \eqref{eq:Q_norm}.

$(b):$ \ It follows from $(a)$ that $\cN$ is a closed subset of $H^1(M)$. The function $\Psi(u):=\mathscr{Q} u-F(u)$ is of class $\cC^2$ on $H^1(M)$ and $\cN=\Psi^{-1}(0)\smallsetminus\{0\}$. Since
\begin{align}\label{eq:psi}
\Psi'(u)u&=2\mathscr{Q}u-F'(u)u=2\mathscr{Q}u-pa_n\im |u|^p\dv-qb\idm |u|^q\d \sigma \nonumber\\
&\leq 2\mathscr{Q}u-qF(u) =(2-q)\mathscr{Q}u<0\qquad\text{for every \ }u\in\cN,
\end{align}
$0$ is a regular value of $\Psi$. This shows that $\cN$ is a Hilbert submanifold of class $\cC^2$ of $H^1(M)$. Furthermore, \eqref{eq:psi} shows that $u$ does not belong to the tangent space $T_u\cN$ of $\cN$ at $u$. Hence, $H^1(M)=T_u\cN\oplus\r u$. So, if $u\in\cN$ and $J'(u)v=0$ for every $v\in T_u\cN$ then, as $J'(u)u=0$, we have that $J'(u)v=0$ for every $v\in H^1(M)$. This shows that every critical point of the restriction of $J$ to $\cN$ is a critical point of $J$, as claimed.

$(c):$ \ Let $u\in H^1(M)\smallsetminus\{0\}$ and $J_u:[0,\infty)\to\r$ be given by
$$J_u(t):=J(tu)=\frac{\mathscr{Q}u}{2} t^2- a_ut^p-b_ut^q,\quad\text{with \ }a_u:=\frac{a_n}{p}\im|u|^p\dv, \ b_u:=\frac{b}{q}\idm|u|^q\d\sigma.$$
As $\mathscr{Q}u>0$ and $a_u>0$, $J_u$ has a unique critical point $t_u$ in $(0,\infty)$ which is a global maximum. Note that, if $t\in(0,\infty)$, then $tu\in\cN$ iff $t_u$ is a critical point of $J_u$. This completes the proof.
\end{proof}

Set
$$\mu_1(M):=\inf_{u\in\cN}J(u).$$

\begin{lemma}\label{lem:positive}
\begin{itemize}
\item[$(i)$]$\mu_1(M)>0$.
\item[$(ii)$]If $J(u)<2\mu_1(M)$ and $J'(u)=0$, then $u$ does not change sign.
\item[$(iii)$]If $\mu_1(M)=J(u)$ for some $u\in\cN$, then $|u|$ is a solution of \eqref{problem}, $|u|$ is smooth up to the boundary and $|u|>0$ on $M$. 
\end{itemize}
\end{lemma}

\begin{proof}
$(i):$ \ Using Lemma \ref{lem:nehari}$(a)$ we see that
$$J(u)\geq\frac{1}{2}\mathscr{Q}u-\frac{1}{q} F(u)=\frac{q-2}{2q}\mathscr{Q}u\geq\frac{q-2}{2q}c_0>0,$$
for all $u\in\cN$, and the statement follows.

$(ii):$ \ Let $u^+:=\max\{u,0\}$ and $u^-:=\min\{u,0\}$. Then $J'(u^\pm)u^\pm=J'(u)u^\pm=0$. If $u^+\neq 0$ and $u^-\neq 0$, then $u^\pm\in\cN$ and, therefore, $J(u)=J(u^+)+J(u^-)\geq 2\mu_1(M)$, a contradiction.

$(iii):$ \ As $\mu_1(M)=J(|u|)$, $|u|$ is a weak solution of \eqref{problem}. The smoothness of $|u|$ follows from standard regularity arguments, and the maximum principle and the boundary point lemma imply that $|u|>0$ on $M$.
\end{proof}

If $u$ is a sign-changing solution to \eqref{problem}, replacing $v$ with $u^+:=\max\{u,0\}$ and $u^-:=\min\{u,0\}$ in \eqref{eq:derivative} yields $J'(u^\pm)u^\pm=0$. Since $u^+\neq 0$ and $u^-\neq 0$, we have that $u$ belongs to the set
$$\cE:=\{u\in\cN:u^+\in\cN\text{ and }u^-\in\cN\}.$$
Let
$$\mu_2(M):=\inf_{u\in\cE}J(u).$$

\begin{proposition} \label{prop:mu is invariant}
The numbers $\mu_1(M)$ and $\mu_2(M)$ are conformal invariants.
\end{proposition}

\begin{proof}
Let $\tilde{g}:=\vp^{4/(n-2)}g$, where $\vp$ is a smooth positive function on $M$. To make it clear which metric is being used we will include it as a subscript and write, for example,
$$\mathscr{Q}_g[u]:=\im(|\nabla u|_g^2+c_nR_gu^2)\dv_g+\idm d_nh_gu^2\d \sigma_g.$$
As shown in \cite[(1.8), (1.9) and (1.10)]{e00}, for every $u\in H^1(M)$,
$$\mathscr{Q}_g[u]=\mathscr{Q}_{\tilde{g}}[\vp^{-1}u],\quad\im|u|^p\d v_g=\im|\vp^{-1}u|^p\d v_{\tilde{g}}\quad\text{and}\quad\idm|u|^q\d\sigma_g=\idm|\vp^{-1}u|^q\d\sigma_{\tilde{g}}$$
Therefore, $J_g(u)=J_{\tilde{g}}(\vp^{-1}u)$ and
$$u\in\cN_g\quad\text{if and only if}\quad \vp^{-1}u\in\cN_{\tilde{g}}.$$
If follows that $\mu_1(M)$ is a conformal invariant. Furthermore, as $(\vp^{-1}u)^\pm=\vp^{-1}u^\pm$,
$$u\in\cE_g\quad\text{if and only if}\quad \vp^{-1}u\in\cE_{\tilde{g}}.$$
Thus, $\mu_2(M)$ is a conformal invariant.
\end{proof}

It follows from Lemma \ref{lem:nehari}$(c)$ that $\mu_1(M)$ coincides with the mountain pass value $J_{mp}$ defined in \cite[p. 812]{hl}. Han and Li showed that this value is attained if $\mu_1(M)<\mu_1(B_b)$, where $B_b$ is the spherical cap defined as follows: Let $\xi_b:=(0,\ldots,0,-\frac{b}{n-2},0)\in\r^{n+1}$, $\sn_b:=\{\xi\in\r^{n+1}:|\xi-\xi_b|=1\}$ be the unit sphere in $\r^{n+1}$ centered at $\xi_b$, and $\pi_b:\sn_b\to\rn$ be the stereographic projection from the north pole $N_b:=(0,\ldots,0,-\frac{b}{n-2},1)\in\r^{n+1}$. Then,
$$B_b:=\pi_b^{-1}(\rn_+),$$
where $\rn_+:=\{(x_1,\ldots,x_n):x_n\geq 0\}$ is the upper half-space in $\rn$. Han and Li showed that the inequality $\mu_1(M)<\mu_1(B_b)$ holds true for any $b\in\r$ if $n\geq 5$ and $M$ has a nonumbilic point on its boundary.
  
\vskip1truecm
\tikzset{every picture/.style={line width=0.75pt}} 
\begin{center}
\resizebox{0.5\textwidth}{!}{
\begin{tikzpicture}[x=0.75pt,y=0.75pt,yscale=-1,xscale=1]
\draw  [draw opacity=0][fill={rgb, 255:red, 237; green, 244; blue, 246 }  ,fill opacity=1 ] (268.6,132.46) .. controls (267.81,133.32) and (266.95,134.15) .. (266.03,134.95) .. controls (251.51,147.47) and (226.13,146.45) .. (209.35,132.66) .. controls (197.18,122.66) and (192.87,108.7) .. (196.96,97) -- (235.65,109.97) -- cycle ; \draw   (268.6,132.46) .. controls (267.81,133.32) and (266.95,134.15) .. (266.03,134.95) .. controls (251.51,147.47) and (226.13,146.45) .. (209.35,132.66) .. controls (197.18,122.66) and (192.87,108.7) .. (196.96,97) ;  
\draw   (105.11,184.39) .. controls (105.11,136.2) and (144.7,97.12) .. (193.53,97.12) .. controls (242.37,97.12) and (281.96,136.2) .. (281.96,184.39) .. controls (281.96,232.59) and (242.37,271.66) .. (193.53,271.66) .. controls (144.7,271.66) and (105.11,232.59) .. (105.11,184.39) -- cycle ;
\draw  [dash pattern={on 0.84pt off 2.51pt}] (105,183.83) .. controls (105,171.77) and (144.39,162) .. (192.98,162) .. controls (241.57,162) and (280.96,171.77) .. (280.96,183.83) .. controls (280.96,195.89) and (241.57,205.66) .. (192.98,205.66) .. controls (144.39,205.66) and (105,195.89) .. (105,183.83) -- cycle ;
\draw  [dash pattern={on 4.5pt off 4.5pt}]  (195.96,182.66) -- (378.96,184.66) ;
\draw [shift={(195.96,182.66)}, rotate = 0.63] [color={rgb, 255:red, 0; green, 0; blue, 0 }  ][fill={rgb, 255:red, 0; green, 0; blue, 0 }  ][line width=0.75]      (0, 0) circle [x radius= 3.35, y radius= 3.35]   ;
\draw  [dash pattern={on 4.5pt off 4.5pt}]  (196.96,98.66) -- (375.96,180.66) ;
\draw [shift={(196.96,98.66)}, rotate = 24.61] [color={rgb, 255:red, 0; green, 0; blue, 0 }  ][fill={rgb, 255:red, 0; green, 0; blue, 0 }  ][line width=0.75]      (0, 0) circle [x radius= 3.35, y radius= 3.35]   ;
\draw    (384,202) -- (378.03,33.66) ;
\draw [shift={(377.96,31.66)}, rotate = 87.97] [color={rgb, 255:red, 0; green, 0; blue, 0 }  ][line width=0.75]    (10.93,-3.29) .. controls (6.95,-1.4) and (3.31,-0.3) .. (0,0) .. controls (3.31,0.3) and (6.95,1.4) .. (10.93,3.29)   ;
\draw    (431.96,124.66) -- (292.22,298.11) ;
\draw [shift={(290.96,299.66)}, rotate = 308.86] [color={rgb, 255:red, 0; green, 0; blue, 0 }  ][line width=0.75]    (10.93,-3.29) .. controls (6.95,-1.4) and (3.31,-0.3) .. (0,0) .. controls (3.31,0.3) and (6.95,1.4) .. (10.93,3.29)   ;
\draw  [dash pattern={on 4.5pt off 4.5pt}]  (105.96,181.66) -- (193.96,182.66) ;
\draw  [draw opacity=0][fill={rgb, 255:red, 237; green, 244; blue, 246 }  ,fill opacity=1 ] (452.35,101.66) -- (543.53,101.66) -- (393.12,287.66) -- (301.95,287.66) -- cycle ;
\draw    (382.96,184.66) -- (548.96,185.65) ;
\draw [shift={(550.96,185.66)}, rotate = 180.34] [color={rgb, 255:red, 0; green, 0; blue, 0 }  ][line width=0.75]    (10.93,-3.29) .. controls (6.95,-1.4) and (3.31,-0.3) .. (0,0) .. controls (3.31,0.3) and (6.95,1.4) .. (10.93,3.29)   ;
\draw [shift={(382.96,184.66)}, rotate = 0.34] [color={rgb, 255:red, 0; green, 0; blue, 0 }  ][fill={rgb, 255:red, 0; green, 0; blue, 0 }  ][line width=0.75]      (0, 0) circle [x radius= 3.35, y radius= 3.35]   ;
\draw    (236.76,106.7) .. controls (274.35,79.41) and (253.56,125.21) .. (292.96,95.66) ;
\draw [shift={(235,108)}, rotate = 323.13] [color={rgb, 255:red, 0; green, 0; blue, 0 }  ][line width=0.75]    (10.93,-3.29) .. controls (6.95,-1.4) and (3.31,-0.3) .. (0,0) .. controls (3.31,0.3) and (6.95,1.4) .. (10.93,3.29)   ;
\draw (390,244.4) node [anchor=north west][inner sep=0.75pt]    {$\mathbb{R}_{+}^{n}$};
\draw (379,188.4) node [anchor=north west][inner sep=0.75pt]    {$0$};
\draw (196,186.4) node [anchor=north west][inner sep=0.75pt]    {$\xi _{b}$};
\draw (185,67.4) node [anchor=north west][inner sep=0.75pt]    {$N_{b}$};
\draw (297,82.4) node [anchor=north west][inner sep=0.75pt]    {$B_{b}$};
\end{tikzpicture}
}\end{center}
\vskip1truecm

Our aim is to give a condition that guarantees that $\mu_2(M)$ is attained at some $u\in\cE$, which is a sign-changing solution of \eqref{problem}.

First, we give a minmax characterization of $\mu_2(M)$. We set
\begin{align*}
\cP^+:=\{u\in H^1(M):u\geq 0\},\qquad \cP^-:=\{u\in H^1(M):u\leq 0\},
\end{align*}
and define $\tau:H^1(M)\to[0,\infty)$ by $\tau(0):=0$ and
\begin{equation*}
\tau(u):=
\begin{cases}
\frac{F(u)}{\mathscr{Q}u} &\text{ \ if \ }u\neq 0\text{ \ and \ }b\geq 0,\smallskip \\
\frac{F_1(u)}{\mathscr{Q}u-F_2(u)} &\text{ \ if \ }u\neq 0\text{ \ and \ }b<0,
\end{cases}
\end{equation*}
where $F_1(u):=a_n\im |u|^p \dv$, $F_2(u):=b\idm|u|^q\d\sigma$ and $F=F_1+F_2$, as defined in \eqref{eq:F}. The function $\tau$ is continuous and, as $\mathscr{Q}u=\mathscr{Q}u^++\mathscr{Q}u^-$ and $F_i(u)=F_i(u^+)+F_i(u^-)$, $i=1,2$, one has that
\begin{equation}\label{eq:tau}
\tau(u)\leq\tau(u^+)+\tau(u^-)\qquad\text{for any \ }u\in H^1(M).
\end{equation}
Consider the set $\Sigma$ of all continuous maps $\sigma:[0,1]\times[0,1]\to H^1(M)$ that satisfy
\begin{equation*}
\begin{cases}
\sigma(0,t)\in\cP^+, &\sigma(1,t)\in\cP^-,\\
\sigma(s,0)=0, &J(\sigma(s,1))\leq 0\text{ \ and \ }\tau(\sigma(s,1))\geq 2.
\end{cases}
\end{equation*}

\begin{lemma}\label{lem:minmax}
\begin{itemize}
\item[$(a)$]Given $u\in\cE$ there exists $\sigma_u\in\Sigma$ such that $\max_{s,t\in I} J(\sigma_u(s,t))=J(u)$.
\item[$(b)$] If $\sigma\in\Sigma$ then there exists $(s_0,t_0)\in(0,1)\times(0,1)$ such that $\sigma(s_0,t_0)\in\cE$.
\end{itemize}
\smallskip

As a consequence,
$$\mu_2(M)=\inf_{\sigma\in\Sigma}\,\max_{s,t\in[0,1]}J(\sigma(s,t)).$$
\end{lemma}

\begin{proof}
$(a):$ \ Let $u\in\cE$ and set $u_s:=(1-s)u^++su^-$, $s\in[0,1]$. Note that $u_s\neq 0$. Define $\gamma:[0,1]\to\cN$ by $\gamma(s):=t_{u_s}u_s$ with $t_{u_s}$ as in Lemma \ref{lem:nehari}$(c)$, that is, $t_{u_s}$ is the unique positive number such that $t_{u_s}u_s\in\cN$. Since $u,u^+,u^-\in\cN$ and $u_\frac{1}{2}=\frac{1}{2}u$ we have that
$$\gamma(0)=u^+, \qquad \gamma(1)=u^-\qquad \text{and}\qquad \gamma(\tfrac{1}{2})=u.$$
Fix $R>0$ such that $J(R\gamma(s))\leq 0$ and $\tau(R\gamma(s))\geq 2$ for all $s\in[0,1]$, and define $\sigma_u:[0,1]\times[0,1]\to H^1(M)$ by
$$\sigma_u(s,t):=Rt\gamma(s).$$
Since $\gamma(s)\in\cN$ we have, by Lemma \ref{lem:nehari}$(c)$, that $J(Rt\gamma(s))\leq J(\gamma(s))$ for all $s,t\in[0,1]$ and, as $u^\pm\in\cN$,
$$J(\gamma(s))=J(\gamma(s)^+)+J(\gamma(s)^-)=J(t_{u_s}(1-s)u^+)+J(t_{u_s}su^-)\leq J(u^+)+J(u^-)=J(u)$$
for all $s\in[0,1]$. This shows that $J(u)=\max_{s,t\in I} J(\sigma_u(s,t))$ for every $u\in\cE$ and, as a consequence, 
$$\mu_2(M)\geq\inf_{\sigma\in\Sigma}\,\max_{s,t\in[0,1]}J(\sigma(s,t)).$$

$(b):$ \ Let $\sigma\in\Sigma$ and define $f=(f_1,f_2):[0,1]\times[0,1]\to\r^2$ by
$$f_1(s,t):=\tau(\sigma(s,t)^+)-\tau(\sigma(s,t)^-),\qquad f_2(s,t):=\tau(\sigma(s,t)^+)+\tau(\sigma(s,t)^-)-2.$$
These are continuous functions and, for all $s,t\in [0,1]$, they satisfy
$$f_1(0,t)\geq 0,\qquad f_1(1,t)\leq 0, \qquad f_2(s,0)=-2,\qquad f_2(s,1)\geq 0,$$
where the last inequality follows from \eqref{eq:tau}. By the Poincaré-Miranda theorem \cite{k} there exists $(s_0,t_0)\in[0,1]\times[0,1]$ such that $f_1(s_0,t_0)=0=f_2(s_0,t_0)$. This implies that 
$$\tau(\sigma(s_0,t_0)^+)=1=\tau(\sigma(s_0,t_0)^-).$$
Noting that $\tau(u)=1$ if and only if $u\in\cN$, we conclude that $\sigma(s_0,t_0)\in\cE$. As a consequence, 
$$\mu_2(M)\leq\inf_{\sigma\in\Sigma}\,\max_{s,t\in[0,1]}J(\sigma(s,t)).$$
This completes the proof.
\end{proof}

By \eqref{eq:Q_norm}, the expressions
$$\langle u,v \rangle_\mathscr{Q}:=\im(\nabla u\cdot\nabla v+c_nRuv)\dv+\idm d_nhuv\d \sigma, \qquad \|u\|_\mathscr{Q}:=\sqrt{\mathscr{Q}u},$$
define an inner product and a norm in $H^1(M)$, equivalent to the standard one. The gradient $\nabla J(u)$ of $J$ at a point $u\in H^1(M)$ with respect to this inner product is $\nabla J(u)=u-K_1(u)-K_2(u)$ where $K_i(u)\in H^1(M)$ is defined by
\begin{equation}\label{eq:defK}
\langle K_1(u),v \rangle_\mathscr{Q}=a_n\im|u|^{p-2}uv\dv,\qquad\langle K_2(u),v \rangle_\mathscr{Q}=b\idm|u|^{q-2}uv\d\sigma,
\end{equation}
for all $v\in H^1(M)$. The negative gradient flow $\eta:\mathcal{G}\to H^1(M)$ of $J$ is the solution of the Cauchy problem
\begin{equation*}
\begin{cases}
\frac{\d}{\d t}\eta(t,u)= - \nabla J(\eta(t,u)),\\
\eta(0,u)= u,
\end{cases}
\end{equation*}
where $\mathcal{G}:= \{ (t,u): u \in H^1(M), \ 0 \leq t < T(u) \}$ and $T(u) \in ( 0, \infty]$ is the maximal existence time for the trajectory $t \to \eta(t,u)$. A subset $\cZ$ of $H^1(M)$ is said to be \emph{strictly positively invariant} under $\eta$ if 
$$\eta(t,u) \in \mathrm{int} (\cZ) \quad \text{for every \ } u \in \cZ \text{ and every } t \in (0,T(u)),$$
where $\mathrm{int}(\mathcal{Z})$ denotes the interior of $\mathcal{Z}$ in $H^1(M)$. If $\cZ$ is strictly positively invariant under $\eta$, then the set
$$\cA(\cZ):= \{ u \in H^1(M): \eta(t,u) \in \cZ\text{ for some } t \in (0,T(u)) \}$$
is open in $H^1(M)$ and the entrance time map $e_{\cZ}: \cA(\cZ)\to\r$, defined by
$$e_{\cZ}(u):= \inf \{ t \geq 0:\eta(t,u)\in\cZ\},$$
is continuous. 

Given $\alpha>0$ we set
$$\dist_\mathscr{Q}(u,\cZ):=\inf_{v\in\cZ}\|u-v\|_\mathscr{Q}\qquad\text{and}\qquad B_\alpha(\cZ):=\{u\in H^1(M):\dist_\mathscr{Q}(u,\cZ)\leq\alpha\}.$$
and, for $d\in\r$, we write 
$$J^{\leq d}:=\{u\in H^1(M):J(u)\leq d\}.$$

The proof of the next lemma is inspired by \cite[Lemma 2]{cw}.

\begin{lemma} \label{lem:invariance}
There exists $\alpha>0$ such that
\begin{itemize}
\item[$(a)$] $B_\alpha(\cP^+)\cap\cE=\emptyset$ \ and \ $B_\alpha(\cP^-)\cap\cE=\emptyset,$
\item[$(b)$] If $b\geq 0$, then $B_\alpha(\cP^+)$ and $B_\alpha(\cP^-)$ are strictly positively invariant. 
\item[$(c)$] If $b\geq 0$ and $J$ does not have a sign-changing critical point $u$ with $J(u)=d$ then the set
$$\cZ_d:=B_\alpha(\cP^+)\cup B_\alpha(\cP^-)\cup J^{\leq d}$$
is strictly positively invariant. As a consequence, the map $\vr_d:\cA(\cZ_d)\to\cZ_d$ given by
$$\vr_d(u):=\eta(e_{\cZ_d},u)$$
is continuous and satisfies $\vr_d(u)=u$ for every $u\in \cZ_d$.
\end{itemize}
\end{lemma}

\begin{proof}
$(a):$ \ Since $M$ is positive, as shown by Escobar in \cite{e1} and \cite{e2}, one has that
\begin{equation}\label{eq:escobarQ}
Q(M):=\inf_{\substack{u\in H^1(M) \\ u\neq 0}}\frac{\mathscr{Q}u}{\big(\im |u|^p\dv\big)^{2/p}}>0\quad\text{and}\quad
Q(M,\partial M):=\inf_{\substack{u\in H^1(M) \\ u|_{\partial M}\neq 0}}\frac{\mathscr{Q}u}{\Big(\idm |u|^q\d \sigma\Big)^{2/q}}>0.
\end{equation}
Therefore, for any $u\in H^1(M)$,
\begin{align}\label{eq:p}
\Big(\im |u^-|^p\dv\Big)^{2/p}&=\inf_{v\in\cP^+}\Big(\im |u-v|^p\dv\Big)^{2/p}\leq \inf_{v\in\cP^+}Q(M)^{-1} \mathscr{Q}[u-v]\nonumber\\
&= Q(M)^{-1}\,\dist_\mathscr{Q}(u,\cP^+)^2,
\end{align}
and, similarly,
\begin{align}\label{eq:q}
\Big(\idm |u^-|^q\d \sigma\Big)^{2/q}\leq Q(M,\partial M)^{-1}\,\dist_\mathscr{Q}(u,\cP^+)^2.
\end{align}
If $u\in\cE$, then $u^-\in\cN$ and $F(u^-)\geq \mu_1(M)>0$. Hence, if $a_n\im |u^-|^p\dv\geq b\idm |u^-|^q\d \sigma$, 
\begin{align*}
\big(\mu_1(M)\big)^{2/p}\leq\big(F(u^-)\big)^{2/p}&\leq \Big(2a_n\im |u^-|^p\dv\Big)^{2/p}\leq (2a_n)^{2/p}Q(M)^{-1}\,\dist_\mathscr{Q}(u,\cP^+)^2
\end{align*}
and, if $a_n\im |u|^p\dv < b\idm |u|^q\d \sigma$,
\begin{align*}
\big(\mu_1(M)\big)^{2/q}\leq\big(F(u^-)\big)^{2/q}&\leq \Big(2b\idm |u^-|^q\d \sigma\Big)^{2/q}\leq (2b)^{2/q}Q(M,\partial M)^{-1}\,\dist_\mathscr{Q}(u,\cP^+)^2.
\end{align*}
Setting 
\begin{equation*}
\alpha_0^2:=
\begin{cases}
\frac{1}{2}\min\Big\{\Big(\frac{\mu_1(M)}{2a_n}\Big)^{2/p}Q(M),\,\Big(\frac{\mu_1(M)}{2b}\Big)^{2/q}Q(M,\partial M)\Big\} &\text{if \ }b>0,\\
\frac{1}{2}\min\Big\{\Big(\frac{\mu_1(M)}{2a_n}\Big)^{2/p}Q(M) &\text{if \ }b\leq 0,
\end{cases}
\end{equation*}
we have that $\alpha_0<\dist_\mathscr{Q}(u,\cP^+)$ for every $u\in\cE$. This shows that $B_{\alpha_0}(\cP^+)\cap\cE=\emptyset$ and, by symmetry, $B_{\alpha_0}(\cP^-)\cap\cE=\emptyset$.

$(b):$ \ Using \eqref{eq:defK}, \eqref{eq:escobarQ} and \eqref{eq:p} we obtain
\begin{align*}
&\dist_\mathscr{Q}(K_1(u),\cP^+)\|K_1(u)^-\|_\mathscr{Q}\leq\|K_1(u)^-\|_\mathscr{Q}^2=\langle K_1(u),K_1(u)^-\rangle_\mathscr{Q}\\
&=a_n\im|u|^{p-2}uK_1(u)^- \leq a_n\im|u^-|^{p-2}u^-K_1(u)^-\\
&\leq a_n\Big(\im |u^-|^p\dv\Big)^{(p-1)/p}\Big(\im |K_1(u)^-|^p\dv\Big)^{1/p}\leq a_nQ(M)^{-p/2}\dist_\mathscr{Q}(u,\cP^+)^{p-1}\|K_1(u)^-\|_\mathscr{Q}.
\end{align*}
Therefore,
\begin{equation*}
\dist_\mathscr{Q}(K_1(u),\cP^+)\leq a_nQ(M)^{-p/2}\,\dist_\mathscr{Q}(u,\cP^+)^{p-1}.
\end{equation*}
If $b\geq 0$, a similar argument, using \eqref{eq:defK},  \eqref{eq:escobarQ} and \eqref{eq:q}, yields
\begin{equation*}
\dist_\mathscr{Q}(K_2(u),\cP^+)\leq b\,Q(M,\partial M)^{-q/2}\,\dist_\mathscr{Q}(u,\cP^+)^{q-1}.
\end{equation*}
As a consequence, there exist $\alpha\in(0,\alpha_0)$ and $\delta\in(0,\frac{1}{2})$ such that
\begin{equation*}
\dist_\mathscr{Q}(K_i(u),\cP^+)\leq \delta\,\dist_\mathscr{Q}(u,\cP^+)\qquad\text{for all \ }u\in B_\alpha(\cP^+),\quad i=1,2,
\end{equation*}
and, setting $K:=K_1+K_2$ we conclude that
\begin{equation} \label{eq:K(u)}
\dist_\mathscr{Q}(K(u),\cP^+)\leq 2\delta\,\dist_\mathscr{Q}(u,\cP^+)\qquad\text{for all \ }u\in B_\alpha(\cP^+).
\end{equation}
Since $B_\alpha(\cP^+)$ is closed and convex, it follows from \cite[Theorem 5.2]{d} that
\begin{equation}\label{eq:positive invariance}
\eta(t,u)\in B_\alpha(\cP^+)\qquad\text{for every \ }u\in B_\alpha(\cP^+)\text{ \ and \ }t\in[0,T(u)).
\end{equation}
Now, arguing by contradiction, assume that $\eta(t_0,u)\in \partial(B_\alpha(\cP^+))$ for some $t_0\in(0,T(u))$ and $u\in B_\alpha(\cP^+)$. By the Mazur separation theorem there exist a continuous linear functional $L:H^1(M)\to\r$ and $\beta>0$ such that $L(\eta(t_0,u))=\beta$ and $L(v)>\beta$ for all $v\in\mathrm{int}(B_\alpha(\cP^+))$. As \eqref{eq:K(u)} implies that $K(\eta(t_0,u))\in\mathrm{int}(B_\alpha(\cP^+))$, we have that
$$\frac{d}{dt}\Big|_{t=t_0}L(\eta(t,u))=L(-\nabla J(\eta(t_0,u))=L(K(\eta(t_0,u))-L(\eta(t_0,u))=L(K(\eta(t_0,u))-\beta>0.$$
But then $L(\eta(t,u))<\beta$ for all $t\in(t_0-\eps,t_0)$ and some $\eps>0$. Thus, $\eta(t,u)\notin B_\alpha(\cP^+)$ for such $t$, contradicting \eqref{eq:positive invariance}. This shows that $B_\alpha(\cP^+)$ is strictly positively invariant and, by symmetry, so is $B_\alpha(\cP^-)$.

$(c)$ is an immediate consequence of $(b)$. 
\end{proof}

\begin{lemma} \label{lem:ps_sequence}
Assume $b\geq 0$ and let $\alpha>0$ be as in \emph{Lemma \ref{lem:invariance}}. Then there exist $u_k\in H^1(M)$ such that 
$$\dist_\mathscr{Q}(u_k,\,\cP^+\cup\cP^-)\geq\alpha, \qquad J(u_k)\to\mu_2(M)\qquad\text{and}\qquad J'(u_k)\to 0.$$
\end{lemma}

\begin{proof}
Set $c:=\mu_2(M)$. Arguing by contradiction, we assume the statement is not true. Then, there exists $0<\eps<c$ such that
\begin{equation}
\|\nabla J(u)\|_\mathscr{Q}\geq \eps\quad\text{for every \ }u\in J^{-1}[c-\eps,c+\eps]\smallsetminus \mathrm{int}(B_\alpha(\cP^+)\cup B_\alpha(\cP^-)).  
\end{equation}
It follows that, for each $u\in J^{-1}[c-\eps,c+\eps]\smallsetminus \mathrm{int}(B_\alpha(\cP^+)\cup B_\alpha(\cP^-))$, there exists $t\in[0,T(u))$ such that 
$\eta(t,u)\in\partial\cZ_{c-\eps}$, where $\cZ_d:=B_\alpha(\cP^+)\cup B_\alpha(\cP^-)\cup J^{\leq d}$. Hence, $\cZ_{c+\eps}\subset\cA(\cZ_{c-\eps})$ and, by Lemma \ref{lem:invariance}$(c)$, the map
$$\vr_{c-\eps}:\cZ_{c+\eps}\to\cZ_{c-\eps},\qquad \vr_{c-\eps}(u):=\eta(e_{\cZ_{c-\eps}},u),$$
is well defined and continuous and satisfies $\vr_{c-\eps}(u)=u$ if $u\in\cZ_{c-\eps}$.

Fix $u\in\cE$ such that $J(u)\leq c+\eps$ and take $\sigma_u\in \Sigma$ as in Lemma \ref{lem:minmax}$(a)$. As $J(\sigma_u(s,t))\leq c+\eps$, the map $\sigma:=\vr_{c-\eps}\circ\sigma_u$ is well defined and continuous and, since $\vr_{c-\eps}(u)=u$ for every $u\in\cP^+\cup \cP^-\cup J^0$, we have that $\sigma\in\Sigma$. Thus, by Lemma \ref{lem:minmax}$(b)$, there exists $(s_0,t_0)\in[0,1]\times[0,1]$ such that $\sigma(s_0,t_0)\in\cE$. But Lemma \ref{lem:invariance}$(a)$ and the definition of $c:=\mu_2(M)$ imply that $\sigma(s,t)\in\cZ_{c-\eps}\subset H^1(M)\smallsetminus\cE$ for all $(s,t)\in [0,1]\times[0,1]$. This is a contradiction. 
\end{proof}

\begin{lemma} \label{lem:hl}
Let $(u_k)$ be a sequence in $H^1(M)$ such that $u_k\rh 0$ weakly but not strongly in $H^1(M)$, $J(u_k)\to c<\mu_1(M)+\mu_1(B_b)$ and $J'(u_k)\to 0$ in $H^{-1}(M)$. Then, $c\geq\mu_1(B_b)$ and, after passing to a subsequence, $\dist_\mathscr{Q}(u_k,\,\cP^+\cup\cP^-)\to0$.
\end{lemma}

We postpone the proof of this lemma to Appendix \ref{app:lemma hl}.

\begin{theorem} \label{thm:mu_2}
If $b\geq 0$ and
$$\mu_2(M)<\mu_1(M)+\mu_1(B_b),$$ 
then there exists $u\in\cE$ such that $J(u)=\mu_2(M)$ and $J'(u)=0$, that is, $u$ is a least energy sign-changing solution of \eqref{problem}.
\end{theorem} 

\begin{proof}
By Lemma \ref{lem:ps_sequence} there is a sequence $(u_k)$ in $H^1(M)$ such that 
\begin{equation} \label{eq:ps}
\dist_\mathscr{Q}(u_k,\,\cP^+\cup\cP^-)\geq\alpha, \qquad J(u_k)\to\mu_2(M)\qquad\text{and}\qquad J'(u_k)\to 0.
\end{equation}
Then,
$$\mu_2(M)+o(1)+o(1)\|u_k\|_\mathscr{Q}\geq J(u_k)-\frac{1}{q} J'(u_k)u_k\geq\frac{q-2}{2q}\mathscr{Q}u_k=\frac{q-2}{2q}\|u_k\|_\mathscr{Q}^2.$$
It follows that $(u_k)$ is bounded in $H^1(M)$ and, passing to a subsequence, $u_k\rh u$ weakly in $H^1(M)$. A standard argument shows that
$$J'(u_k)\vp\to J'(u)\vp\qquad\text{for all \ }\vp\in \cC^\infty(M),$$
and, as  and $J'(u_k)\to 0$ in $H^{-1}(M)$, we have that $J'(u)=0$. It is also standard to verify (see, for instance, \cite[Section 8.3]{w}) that
\begin{align*}
J(u_k)&=J(u_k-u)+J(u)+o(1),\\
o(1)=J'(u_k)&=J'(u_k-u)+J'(u) + o(1)=J'(u_k-u)+o(1)\quad\text{in \ }H^{-1}(M).
\end{align*}
Set $v_k:=u_k-u$. Then $v_k\rh 0$ weakly in $H^1(M)$, $J(v_k)\to c-J(u)$ and $J'(v_k)\to 0$ in $H^{-1}(M)$. Assume first that $(v_k)$ does not converge strongly to $0$. We consider two cases: If $u\neq 0$, then $u\in\cN$. Hence, $J(v_k)\to c-J(u)\leq c-\mu_1(M)<\mu_1(B_b)$, contradicting Lemma \ref{lem:hl}. If, on the other hand, $u= 0$, then $v_k=u_k$. So $\dist_\mathscr{Q}(v_k,\,\cP^+\cup\cP^-)\geq\alpha>0$, which again contradicts Lemma \ref{lem:hl}. 

This shows that $u_k\to u$ strongly in $H^1(M)$. Therefore, $J(u)=\mu_2(M)$, $u\notin\cP^+\cup\cP^-$ and $J'(u^\pm)[u^\pm]=J'(u)[u^\pm]=0$. The last two statements say that $u\in\cE$. 
\end{proof}

Theorem \ref{thm:mu_2} yields the following existence criterion.

\begin{corollary}\label{cor:mu_2a}
Let $b\geq 0$. If there exist $u_0,u_1\in H^1(M)$ such that $u_0,u_1\geq 0$, $u_s:=(1-s)u_0-su_1$ is nontrivial for all $s\in[0,1]$ and 
$$J(ru_s)< \mu_1(M)+\mu_1(B_b)\qquad\text{for all \ }s\in[0,1]\text{ \ and \ }r\geq 0,$$ 
then the problem \eqref{problem} has a least energy sign-changing solution.
\end{corollary}

\begin{proof}
Let $R>0$ and set $\sigma(s,t):=Rtu_s$. Then, $\sigma(0,t)\in\cP^+$ and $\sigma(1,t)\in\cP^-$ for all $t\in[0,1]$, and $\sigma(0,0)=0$. Since $u_s$ is nontrivial,
\begin{align*}
&J(Ru_s)=\frac{R^2}{2}\mathscr{Q}[u_s]-\frac{R^pa_n}{p}\im|u_s|^p\dv-\frac{R^qb}{q}\idm|u_s|^q\d \sigma\to-\infty\qquad\text{as \ }R\to\infty,
\end{align*}
for each $s\in[0,1]$. So we may choose $R>0$ large enough so that $J(\sigma(s,1))\leq 0$ and $\tau(\sigma(s,1))\geq 2$ for all $s\in[0,1]$. Then, $\sigma\in\Sigma$ and, by Lemma \ref{lem:minmax}$(b)$, there exists $(s_0,t_0)\in(0,1)\times(0,1)$ such that $J(\sigma(s_0,t_0))\in\cE$. Using the hypothesis, we get
$$\mu_2(M)\leq J(\sigma(s_0,t_0))< \mu_1(M)+\mu_1(B_b),$$
and the statement follows from Theorem \ref{thm:mu_2}.
\end{proof}

\section{The energy of the test function}
\label{sec:test function}

We assume that $n\geq 7$ and that $M$ has a nonumbilic point $\xi\in\partial M$. As before, we set $p:=\frac{2n}{n-2}$ and $q:=\frac{2(n-1)}{n-2}$.

Let $\rn_+:=\{x=(x',x_n)\in\r^{n-1}\times\r\equiv\rn:x_n>0\}$. The \emph{bubble} 
$$U(x):=\left(\frac{1}{1+|x'|^2+|x_n+\frac{b}{n-2}|^2}\right)^{(n-2)/2}$$
is a solution to the problem
\begin{equation*}
\begin{cases}
-\Delta u = n(n-2)u^{p-1}&\text{in \ }\rn_+ \\
\frac{\partial u}{\partial\nu}=-bu^{q-1} &\text{on \ }\partial\rn_+.
\end{cases}
\end{equation*}
Since the spherical cap $B_b$ (introduced after Proposition \ref{prop:mu is invariant}) is conformally equivalent to $\rn_+$ via the stereographic projection $\pi_b$ and $\mu_1$ is a conformal invariant, we have that
$$\mu_1(B_b)=J(U)=\frac{1}{2}\int_{\rn_+}|\nabla U|^2\d x -\frac{a_n}{p}\int_{\rn_+}U^p\d x-\frac{b}{q}\int_{\partial\rn_+}U^q\d s.$$

As in Escobar \cite[Section 3]{e1}, after a conformal change of metric, we may assume that $h(\xi)=0$ and $R_{ij}(\xi)=0$. We take normal coordinates $x=(x_1,\ldots,x_n)$ around $\xi$ such that $\nu(\xi)=-\frac{\partial}{\partial x_n}$ and the second fundamental form of $\partial M$ at $\xi$ has diagonal form. The diagonal elements are the principal curvatures $\lambda_1,\ldots,\lambda_{n-1}$ and $\partial M$ is given locally around $\xi$ by the equation
$$x_n=\frac{1}{2}\sum_{i=1}^{n-1}\lambda_ix_i^2 + O(|x|^3).$$
Since $\xi$ is not umbilic, not all principal curvatures are zero. Following Han and Li \cite{hl} we consider a test function of the form
\begin{equation*}
W_\eps(x):=\eps^{-\frac{n-2}{2}}\psi(x)\Big[U\Big(\frac{x}{\eps}\Big)+\delta\phi\Big(\frac{x}{\eps}\Big)\Big],
\end{equation*} 
where $\eps$ and $\delta$ are small parameters, $\psi$ is a suitable cut-off function that takes the value $1$ near $0$ and $\phi\in\cC^\infty_c(\rn_+)$. It is shown in \cite{hl} that, for $n\geq 5$, 
\begin{equation*} 
\max_{t\in[0,1]}J(tW_\eps)=\mu_1(B_b)+Q_1(\phi)\eps\delta + Q_2(\phi)\delta^2 + Q_3\eps^2 + o(\eps^2+\delta^2),
\end{equation*}
where $Q_1(\phi)$ is a linear function of $\phi$, $Q_2(\phi)$ is a quadratic funtion of $\phi$ and $Q_3$ is a number, explicitely given in \cite[(2.4), (2.5) and (2.9)]{hl}. They do not depend on $\eps,\delta$, and are such that $Q_2(\phi)\geq 0$ for any choice of $\phi$ and $Q_3>0$ if $b>0$. Han and Li showed that $\phi$ can be chosen to satisfy
$$Q_1^2(\phi)-4Q_2(\phi)Q_3>0.$$
For such $\phi$, we may take $\kappa\in\r$ such that $C_0:=-(Q_1(\phi)\kappa + Q_2(\phi)\kappa^2 + Q_3)>0$, and setting $\delta:=\kappa\eps$ we get
\begin{align} \label{eq:estimate W}
\max_{t\in[0,1]}J(tW_\eps)&=\mu_1(B_b)+(Q_1(\phi)\kappa + Q_2(\phi)\kappa^2 + Q_3)\eps^2 + o(\eps^2+\delta^2)\nonumber\\
&=\mu_1(B_b)-C_0\eps^2+ o(\eps^2)
\end{align}
for all small enough $\eps>0$. This estimate yields the existence of a smooth positive solution $u_0\in\cN$ to \eqref{problem} that satisfies
\begin{equation}\label{eq:J(u_0)}
J(u_0)=\mu_1(M),
\end{equation}
as shown by Han and Li in \cite{hl}.

\begin{proof}[Proof of Theorem \ref{thm:main}]
Let $b\geq 0$. We will apply Corollary \ref{cor:mu_2a} to the functions $u_0$ and $W_\eps$ defined above. First we observe that, since $\phi$ has compact support and recalling that $\delta:=\kappa\eps$, we have that $W_\eps\ge0,$ if $\eps$ is sufficiently small (see \cite[p. 822]{hl}). Next, we will show that, for some $\eps>0$ small enough,
\begin{equation}\label{eq:claim 1}
J(su_0-tW_\eps)<\mu_1(M)+\mu_1(B_b)\qquad\text{for all \ }s,t\in[0,r] \text{ with } s+t=r\text{ \ and all \ }r\in[0,R],
\end{equation}
where $R>0$ is chosen large enough so that $J(su_0-tW_\eps)\leq 0$ for all $s,t\in[0,r]$ with  $s+t=r\geq R$ and every $\eps>0$ sufficiently small.

Since $u_0$ solves \eqref{problem}, it follows from Lemma \ref{lem:nehari}$(c)$, and estimates \eqref{eq:estimate W} and \eqref{eq:J(u_0)} that, for $\eps$ sufficiently small,
\begin{align} \label{eq:energy}
&J(su_0-tW_\eps)=\frac{1}{2}\mathscr{Q}[su_0-tW_\eps]-\frac{a_n}{p}\im|su_0-tW_\eps|^p\dv-\frac{b}{q}\idm|su_0-tW_\eps|^q\d\sigma \nonumber \\
&=J(su_0)+J(tW_\eps)-\langle su_0,tW_\eps\rangle_\mathscr{Q}-\frac{a_n}{p}\im\Big(|su_0-tW_\eps|^p-|su_0|^p-|tW_\eps|^p\Big)\dv\nonumber\\
&\qquad-\frac{b}{q}\idm\Big(|su_0-tW_\eps|^q-|su_0|^q-|tW_\eps|^q\Big)\d\sigma  \nonumber \\
&\leq \mu_1(M)+\mu_1(B_b)-C_0\eps^2+o(\eps^2)-\langle su_0,tW_\eps\rangle_\mathscr{Q}-\frac{a_n}{p}\im\Big(|su_0-tW_\eps|^p-|su_0|^p-|tW_\eps|^p\Big)\dv\nonumber\\
&\qquad
-\frac{b}{q}\idm\Big(|su_0-tW_\eps|^q-|su_0|^q-|tW_\eps|^q\Big)\d\sigma\nonumber \\
&\leq  \mu_1(M)+\mu_1(B_b)-C_0\eps^2+o(\eps^2) +C_1R^q
\im\Big(u_0^{p-1}(U_\eps +\delta|\phi_\eps|)+u_0\left(U_\eps^{p-1}+\delta^{p-1}|\phi_\eps|^{p-1}\right)\Big)\dv\nonumber\\
&\qquad+C_1R^q
\idm\Big(u_0^{q-1}(U_\eps +\delta|\phi_\eps|) +u_0\left(U_\eps^{q-1}+\delta^{q-1}|\phi_\eps|^{q-1}\right)\Big)\d\sigma.
\end{align}
Here used the inequality
$$|a-b|^r-a^r-b^r\geq-r(ab^{r-1}+a^{r-1}b)\qquad\text{for all \ }a,b\geq 0\text{ \ and \ } r\geq 2,$$ 
and we have written $W_\eps$ as
$$W_\eps(x)=  \underbrace{\frac 1{\eps^{n-2\over2}}\psi(x)U\left(\frac x\eps\right)}_{=:U_\eps(x)}+ \delta \underbrace{\frac 1{\eps^{n-2\over2}}\psi(x)\phi\left(\frac x\eps\right)}_{=:\phi_\eps(x)}.$$
To obtain the conclusion we must estimate the last two integrals. In the following, $C$ will denote different positive constants independent of $\eps$. We obtain
\begin{align*}
\im u_0U_\eps^{p-1}\d\sigma &\leq C\int_{\r^{n }}U_\eps(x )^{p-1}\d x\leq   C\int_{\r^{n }}\left(\frac{\eps}{\eps^2+|x|^2}\right)^{(n+2)/2}\d x  \\
& \leq C\int_{\r^{n }}\left(\frac{1}{\eps(1+|y|^2)}\right)^{(n+2)/2}\eps^{n }\d y \leq C\eps^{(n-2)/2}=o(\eps^2)\quad\text{if \ }n\geq 7.
\end{align*}
\begin{align*}
\im u_0^{p-1}U_\eps\d\sigma &\leq C\int_{B(0,\rho)}U_\eps(x ) \d x\leq   C\int_{B(0,\rho)}\left(\frac{\eps}{\eps^2+|x|^2}\right)^{(n-2)/2}\d x \\
& \leq C\eps^{(n-2)/2}\int_{B(0,\rho)} \frac{1}{|x|^{n-2}}\d x \leq C\eps^{(n-2)/2}=o(\eps^2)\quad\text{if \ }n\geq 7.
\end{align*}
Next we estimate the boundary integrals using Fermi coordinates $(x,t)\in\mathbb R^{n-1}\times (0,\infty)$ at $\xi$. 
\begin{align*}
\idm u_0U_\eps^{q-1}\d\sigma &\leq C\int_{\r^{n-1}}U_\eps(x,0)^{q-1}\d x\leq C\int_{\r^{n-1}}\left(\frac{\eps}{\eps^2+|x|^2}\right)^{n/2}\d x \\
& \leq C\int_{\r^{n-1}}\left(\frac{1}{\eps(1+|y|^2)}\right)^{n/2}\eps^{n-1}\d y \leq C\eps^{(n-2)/2}=o(\eps^2)\quad\text{if \ }n\geq 7.
\end{align*}
\begin{align*}
\im u_0^{p-1}U_\eps\d\sigma &\leq C\int_{B(0,\rho)\cap \partial \mathbb R^n_+}U_\eps(x,0 ) \d x\leq   C\int_{B(0,\rho)\cap \partial \mathbb R^n_+} \left(\frac{\eps}{\eps^2+|x|^2}\right)^{(n-2)/2}\d x \\
& \leq C\eps^{(n-2)/2}\int_{B(0,\rho)\cap \partial \mathbb R^n_+} \frac{1}{|x|^{n -2}}\d x \leq C\eps^{(n-2)/2}=o(\eps^2)\quad\text{if \ }n\geq 7.
\end{align*}
Finally, we estimate the terms which involve $\phi_\eps$. Since $\phi$ has compact support we have for some $C>0$
$$\im \phi_\eps(x)^\alpha \d x=\int\limits_{\{|x/\eps|\le c\}}
\left|\frac1{\eps^{n-2\over2}}\phi\left(\frac x\eps\right)\right|^\alpha dx\leq C
\eps^{n-\alpha{n-2\over2}}
\int\limits_{\{|y|\le c\}}
\left| \phi\left(y\right)\right|^\alpha dy\leq  C
\eps^{n-\alpha{n-2\over2}}
$$
and
$$\idm \phi_\eps(x)^\alpha \d x=\int\limits_{\{|x/\eps|\le c\}\cap\partial\mathbb R^n_+}
\left|\frac1{\eps^{n-2\over2}}\phi\left(\frac x\eps\right)\right|^\alpha dx\leq C
\eps^{n-1-\alpha{n-2\over2}}
\int\limits_{\{|y|\le c\}\cap\partial\mathbb R^n_+}
\left| \phi\left(y\right)\right|^\alpha dy\leq  C
\eps^{n-1-\alpha{n-2\over2}}.
$$
Therefore
\begin{align*}
&\im\Big(u_0^{p-1}  \delta|\phi_\eps|+u_0  \delta^{p-1}|\phi_\eps|^{p-1} \Big)\dv + 
\idm\Big(u_0^{q-1} \delta|\phi_\eps| +u_0 \delta^{q-1}|\phi_\eps|^{q-1} \Big)\d\sigma\\
&\leq C \left(\delta\eps^{n+2\over2}+\delta^{n+2\over n-2}\eps^{n-2\over2}
+\delta \eps^{\frac n2}+\delta^{\frac n{n-2}}\eps^{n-2\over2}\right)
 =o(\eps^2).
\end{align*}
Combining the previous estimates with \eqref{eq:energy} we deduce
$$J(su_0-tW_\eps)\leq  \mu_1(M)+\mu_1(B_b)-C_0\eps^2+o(\eps^2) $$
and the claim follows, provided $\eps$ is small enough.
\end{proof}

\appendix

\section{The proof of Lemma \ref{lem:hl}}
\label{app:lemma hl}

The classical (and complete) version of Lemma \ref{lem:hl} is due to Struwe \cite{s}. For closed manifolds Druet, Hebey, and Robert proved a similar result in \cite[Lemma 3.2]{dhr}. The first statement of Lemma \ref{lem:hl} was proven by Han and Li in \cite[Lemma A.2]{hl}. We present a summary of their proof and follow some arguments in \cite{dhr} to obtain the second statement.

As before, we set $p:=\frac{2n}{n-2}$ and $q:=\frac{2(n-1)}{n-2}$.

\begin{proof}[Proof of Lemma \ref{lem:hl}] 
Let $(u_k)$ be a sequence in $H^1(M)$ such that $u_k\rh 0$ weakly but not strongly in $H^1(M)$, $J(u_k)\to c<\mu_1(M)+\mu_1(B_b)$ and $J'(u_k)\to 0$ in $H^{-1}(M)$. Passing to a subsequence we have that $u_k\to0$ strongly in $L^2(M)$ and $u_k|_{\partial M}\to0$ strongly in $L^2(\partial M)$. It follows that
$$0<\kappa_0\leq\|u_k\|^2=\|u_k\|_\mathscr{Q}^2+o(1)=F(u_k)+o(1):=a_n\im |u_k|^p \dv + b\idm|u_k|^q\d\sigma+o(1).$$
Consider the concentration function
$$Q_i(r):=\max_{x\in \overline{M}}\left\{a_n\int_{B_r(x)} |u_i|^p \dv + b\int_{B_r(x)\cap\partial M}|u_i|^q\d\sigma\right\},\qquad r\geq0,$$
where $B_r(x)$ is the open ball in $\overline{M}$ of radius $r$ and center $x$. Given $\kappa\in(0,\kappa_0)$, let $r_i>0$ and $\bar{x}_i\in M$ be such that
\begin{equation} \label{eq:kappa}
\kappa=Q_i(r_i)=a_n\int_{B_{r_i}(\bar{x}_i)} |u_k|^p \dv + b\int_{B_{r_i}(\bar{x}_i)\cap\partial M}|u_k|^q\d\sigma.
\end{equation}
After passing to a subsequence, $\bar{x}_i\to\bar{x}\in \overline{M}$, and fixing $\kappa$ sufficiently small and following the argument in \cite[(A.12) and (A.13)]{hl} one shows that $r_i\to 0$.

Take geodesic normal coordinates at $\bar{x}_i$, set $\tilde{M}_i:=\{z\in\rn:\exp_{\bar{x}_i}(r_iz)\in \overline{M}\text{ and }|r_iz|<\iota_g/2\}$, where $\iota_g$ is the injectivity radius of $M$, and define
$$\tilde{u}_i(z):=r_i^\frac{n-2}{2} u_i(\exp_{\bar{x}_i}(r_iz)),\qquad z\in\tilde{M}_i.$$
Then, after passing to a subsequence, we may choose suitable $R_i<\iota_g/2r_i$ such that $R_i\to\infty$ and
\begin{equation}\label{eq:annulus}
\lim_{i\to\infty}\left[\int_{A_i}\big(|\nabla_{\tilde{g}_i}\tilde{u}_i|^2+|\tilde{u}_i|^p\big)\dv_{\tilde{g}_i}+\int_{A_i\cap\partial'\tilde{M}_i}|\tilde{u}_i|^q\d\sigma_{\tilde{g}_i}\right]=0,
\end{equation}
where $A_i:=\{z\in\tilde{M}_i:R_i<|z|<2R_i\}$, $\partial'\tilde{M}_i:=\{z\in\rn:\exp_{\bar{x}_i}(r_iz)\in\partial M\text{ and }|r_iz|<\iota_g/2\}$ and $\tilde{g}_i$ denotes the metric $(\tilde{g}_i)_{\alpha\beta}(z)=g_{\alpha\beta}(r_iz)$. Next, we choose a smooth cut-off function $\tilde{\eta}_i\in\cC^\infty_c(\rn)$ such that  $\tilde{\eta}_i(z)\in[0,1]$, $\tilde{\eta}_i(z)=1$ if $|z|<R_i$, $\tilde{\eta}_i(z)=0$ if $|z|>2R_i$ and $|\nabla\tilde{\eta}_i(z)|\leq C/R_i$.

Set $\tilde{u}_i^{(1)}:=\tilde{\eta}_i\tilde{u}_i$ and $\tilde{u}_i^{(2)}:=\tilde{u}_i - \tilde{u}_i^{(1)}$ and define
$$u_i^{(1)}(\exp_{\bar{x}_i}(y)):=r_i^{-\frac{n-2}{2}}\tilde{u}_i^{(1)}(y/r_i)=\tilde{\eta}_i(y/r_i)u_i(\exp_{\bar{x}_i}(y)),\qquad u_i^{(2)}:=u_i - u_i^{(1)}.$$
These functions satisfy
\begin{align*}
&u_i^{(1)}\rh 0\text{ \ and \ }u_i^{(2)}\rh 0\quad\text{weakly in \ } H^1(M), \\
&J(u_i)=J(u_i^{(1)})+J(u_i^{(2)})+o(1),\\
&J'(u_i)=J'(u_i^{(1)})+J'(u_i^{(2)})+o(1)\quad \text{in \ }H^{-1}(M), \\
&J'(u_i^{(1)})u_i^{(1)}=o(1)\text{ \ and \ }J'(u_i^{(2)})u_i^{(2)}=o(1);
\end{align*}
see \cite[Proof of Lemma A.2]{hl}. It follows that
$$c^{(1)}:=\lim_{i\to\infty}J(u_i^{(1)})\geq0\text{ \ \ and \ \ }c^{(2)}:=\lim_{i\to\infty}J(u_i^{(2)})\geq0.$$

Let $\tilde{u}^{(1)}$ be the weak limit of $(\tilde{u}_i^{(1)})$ in $H^1_{loc}$. From equations \eqref{eq:kappa} and \eqref{eq:annulus} we see that
\begin{align*}
&\sup_{z\in\tilde{M}_i}\left[\int_{B_1(z)\cap\tilde{M}_i}|\tilde{u}_i^{(1)}|^p + \int_{B_1(z)\cap\partial'\tilde{M}_i}|\tilde{u}_i^{(1)}|^q\right]+o(1)\leq\kappa \\
&\qquad\qquad=\int_{B_1(0)\cap\tilde{M}_i}|\tilde{u}_i^{(1)}|^p + \int_{B_1(0)\cap\partial'\tilde{M}_i}|\tilde{u}_i^{(1)}|^q + o(1).
\end{align*}
A standard argument shows that $\tilde{u}^{(1)}\neq 0$; see \cite[Proof of Lemma A.2]{hl} or \cite[Proof of Lemma 3.2]{dhr}. Given $\tilde{\vp}\in\cC^\infty_c(\rn)$, set $\vp_i(\exp_{\bar{x}_i}(y)):=r_i^{(2-n)/2}\tilde{\vp}(r_i^{-1}y)$. Then,
\begin{align}\label{eq:solution}
o(1)&=J'(u_i)\vp_i=J'(u_i^{(1)})\vp_i + o(1)\|\vp_i\| \nonumber\\
&=\int_{\tilde{M}_i}\Big(\nabla_{\tilde{g}_i}\tilde{u}_i^{(1)}\nabla_{\tilde{g}_i}\tilde{\vp}+c_nR_{\tilde{g}_i}\tilde{u}_i^{(1)}\tilde{\vp}\Big)\dv_{\tilde{g}_i}+d_n\int_{\partial'\tilde{M}_i}h_{\tilde{g}_i}\tilde{u}_i^{(1)}\tilde{\vp}\d\sigma_{\tilde{g}_i}\nonumber \\
&\quad -a_n\int_{\tilde{M}_i}|\tilde{u}_i^{(1)}|^{p-2}\tilde{u}_i^{(1)}\tilde{\vp}\dv_{\tilde{g}_i} - b\int_{\partial'\tilde{M}_i}|\tilde{u}_i^{(1)}|^{q-2}\tilde{u}_i^{(1)}\tilde{\vp}\d\sigma_{\tilde{g}_i} + o(1)\|\vp_i\|.
\end{align}
Let $d:=\lim_{i\to\infty}r_i^{-1}\dist_g(\bar{x}_i,\partial M)$. We consider two cases:

$(1):$ \ If $d\in\r$ then, passing to the limit in \eqref{eq:solution}, we get that
$$J_{\mathbb{H}}'(\tilde{u}^{(1)})\tilde{\vp}=\int_{\mathbb{H}}\nabla\tilde{u}^{(1)}\nabla\tilde{\vp}-a_n\int_{\mathbb{H}}|\tilde{u}^{(1)}|^{p-2}\tilde{u}^{(1)}\tilde{\vp}-b\int_{\partial\mathbb{H}}|\tilde{u}^{(1)}|^{q-2}\tilde{u}^{(1)}\tilde{\vp}=0,$$
where, up to a rotation, $\mathbb{H}:=\{z\in\rn:z_n>-d\}$ and
$$J_{\mathbb{H}}(u):=\frac{1}{2}\int_{\mathbb{H}}|\nabla u|^2-\frac{a_n}{p}\int_{\mathbb{H}}|u|^p-\frac{b}{q}\int_{\partial\mathbb{H}}|u|^q$$ 
is the variational functional for the problem
\begin{equation}\label{eq:prob_H}
\begin{cases}
-\Delta u=a_n|u|^{p-2}u &\text{in \ }\mathbb{H},\\
\frac{\partial u}{\partial z_n}=-b|u|^{q-2}u &\text{on \ }\partial\mathbb{H}.
\end{cases}
\end{equation}
Therefore, $\tilde{u}^{(1)}$ is a nontrivial solution to \eqref{eq:prob_H} in this case. 

$(2):$ \ If $d=\infty$ then, passing to the limit in \eqref{eq:solution}, we obtain
$$J_{\rn}'(\tilde{u}^{(1)})\tilde{\vp}=\irn\nabla\tilde{u}^{(1)}\nabla\tilde{\vp}-a_n\irn|\tilde{u}^{(1)}|^{p-2}\tilde{u}^{(1)}\tilde{\vp}=0,$$
where
$$J_{\rn}(u):=\frac{1}{2}\irn|\nabla u|^2-\frac{a_n}{p}\irn|u|^p$$ is the variational functional for the problem
\begin{equation}\label{eq:prob_rn}
-\Delta u=a_n|u|^{p-2}u\qquad\text{in \ }\rn.
\end{equation}
So, in this case, $\tilde{u}^{(1)}$ is a nontrivial solution to \eqref{eq:prob_rn}. 

Set
$$w_i:=u_i^{(1)}-v_i,\qquad\text{where \ \ }v_i(\exp_{\bar{x}_i}(y)):=r_i^{-\frac{n-2}{2}}\tilde{\eta}_i(y/r_i)\tilde{u}^{(1)}(y/r_i).$$
Arguing as in Step 3 of \cite[Proof of Lemma 3.2]{dhr} one shows that
$$v_i\rh 0\quad\text{weakly in \ }H^1(M),$$
and that, in case (1), i.e., when $d\in\r$,
\begin{align*}
\im|\nabla w_i|^2\dv&=\im|\nabla u_i^{(1)}|^2\dv-\int_{\mathbb{H}}|\nabla\tilde{u}^{(1)}|^2+o(1), \nonumber\\
\im|w_i|^p\dv&=\im|u_i^{(1)}|^p\dv-\int_{\mathbb{H}}|\nabla\tilde{u}^{(1)}|^p+o(1), \nonumber\\
\idm|w_i|^q\dv&=\idm|u_i^{(1)}|^q\dv-\int_{\partial\mathbb{H}}|\nabla\tilde{u}^{(1)}|^q+o(1).
\end{align*}
Therefore,
\begin{align*}
J(w_i)&=J(u_i^{(1)})-J_\mathbb{H}(\tilde{u}^{(1)})+o(1),\nonumber \\
J'(w_i)w_i&=J'(u_i^{(1)})u_i^{(1)}-J_\mathbb{H}'(\tilde{u}^{(1)})\tilde{u}^{(1)}+o(1)=o(1).
\end{align*} 
As a consequence, $J(w_i)\geq 0$ and, recalling our assumption on $c$, we derive
$$\mu_1(B_b)\leq J_\mathbb{H}(\tilde{u}^{(1)})\leq J(u_i^{(1)})+o(1)\leq J(u_i)+o(1)=c<2\mu_1(B_b).$$
In particular, $c\geq\mu_1(B_b)$, as claimed. Furthermore, Lemma \ref{lem:positive} states that $\tilde{u}^{(1)}$ does not change sign. Hence, neither does $v_i$. We claim that, after passing to a subsequence,
\begin{equation*}
\|w_i\|_\mathscr{Q}=\|u_i^{(1)}-v_i\|_\mathscr{Q}\to 0.
\end{equation*}
Indeed, if no subsequence of $(w_i)$ converges strongly to $0$, then, as $J'(w_i)w_i=o(1)$, we would have that $J(w_i)+o(1)\geq \mu_1(M)$. But then, $c\geq\mu_1(M)+\mu_1(B_b)$, contradicting our assumption. A similar argument shows that
\begin{equation*}
\|u_i^{(2)}\|_\mathscr{Q}=\|u_i-u_i^{(1)}\|_\mathscr{Q}\to 0.
\end{equation*}
Therefore, $\|u_i-v_i\|_\mathscr{Q}\to 0$ and, since $v_i\in\cP^+\cup\cP^-$, we have that $\dist_\mathscr{Q}(u_i,\cP^+\cup\cP^-)\to 0$. This completes the proof in case (1).

In case (2) the proof follows the same pattern, one needs only to observe that, for any $b\in\r$, the minimal energy $\mu_1(\sn)$ of a nontrivial solution to \eqref{eq:prob_rn} satisfies $\mu_1(\sn)>\mu_1(B_b)$, as shown in \cite[Lemma A.1]{hl}. 
\end{proof}

\bibliography{escobar}
\bibliographystyle{amsplain}

\end{document}